\documentclass[11 pt]{amsart}
\usepackage{graphicx}
\usepackage[hidelinks]{hyperref}
\usepackage{color}
\usepackage{amssymb}
\usepackage[utf8]{inputenc}
\usepackage{epstopdf}
\usepackage{nicefrac}
\usepackage{amsthm}
\usepackage{amsthm}
\usepackage{tikz}
\usetikzlibrary{arrows,chains,matrix,positioning,scopes}
\makeatletter
\tikzset{join/.code=\tikzset{after node path={%
\ifx\tikzchainprevious\pgfutil@empty\else(\tikzchainprevious)%
edge[every join]#1(\tikzchaincurrent)\fi}}}
\makeatother
\tikzset{>=stealth',every on chain/.append style={join},
         every join/.style={->}}
\tikzstyle{labeled}=[execute at begin node=$\scriptstyle,
   execute at end node=$]
\usepackage{tikz,fullpage}
\usetikzlibrary{patterns}
\numberwithin{equation}{section}
\usepackage[margin=1.27in]{geometry}
\usetikzlibrary{arrows,%
                petri,%
                topaths}%
\usepackage{tkz-berge}
\usepackage{tikz-cd}
\usepackage{mathtools}
\usepackage{amsfonts}
\usepackage[position=top]{subfig}
\usetikzlibrary{%
  matrix,%
  calc,%
  arrows%
}
\newtheorem{thm}{THEOREM}[section]

\newtheorem{cor}[thm]{COROLLARY}

\newtheorem{lemma}[thm]{LEMMA}

\newcommand{\G}{\Gamma}

\begin{document}
\title{Invariants of the $\mathbb Z_2$ orbifolds of the Podle\'s two spheres.}

\author{Safdar Quddus}
\address{School of Mathematical Sciences, NISER, Bhubaneswar, India}
\email{safdar@niser.ac.in}

\date{\today}


\maketitle

\begin{abstract}

There are two $\mathbb Z_2$ orbifolds of the Podle\'s quantum two-sphere, one being the quantum two-disc $D_q$ and other the quantum two-dimensional real projective space $\mathbb R P^{2}_q$. In this article we calculate the Hochschild and cyclic homology and cohomology groups of these orbifolds and also the corresponding Chern-Connes indices. 
\end{abstract}

\bigskip

\section{Introduction} \label{sec-bck}
In noncommutative geometry Podle\'s noncommutative two-spheres \cite{P} are noncommutative low-dimensional manifolds, these are $SU_q(2)$-homogeneous spaces. Denoted by $S^2_{q,s}$ where 
$$ q^n \neq 1 \text{    for all } n\in \mathbb N.$$
and $s \in [0,1]$, the Podle\'s quantum spheres are parametrized by $s$. We shall study the associated C*-algebra which corresponds to the algebra of smooth functions on the quantum two spheres $S^2_{q,s}$. An extensive study of the Dirac operators, spectral triples and corresponding local index formulae for the Podle\'s spheres can be found in articles \cite{CP} \cite{DL}. In the paper \cite{GoJ} authors studied the quantum isometries of the Podle\'s spheres. In \cite{HMS} the authors studied the quantum disc $D_q$ and the two dimensional quantum real projective space $\mathbb RP^2_q$ arising by two different involutive automorphisms of $S^2_{q,1}$(for $s>0, C(S^2_{q,s}) \cong C(S^2_{q,1}$)\cite{Sh}). All automorphisms of $S^2_{q,s}$ are known to be diagonal as described in \cite{K}. Let $\mathcal A$ denote the C*-algebra $C(S^2_{q,s})$ and $\rho \in Aut(\mathcal A)$. Hochschild and cyclic homology of the Podle\'s quantum spheres was calculated by Masuda, Nakagami and Watanabe [MNW], using a free resolution . In \cite{H} author used this resolution to calculate the Hochschild homologies $H_n(\mathcal A, {}_{\rho}\mathcal A)$, which by \cite{HK} are isomorphic to the twisted Hochschild homologies $HH^{\rho}_n(\mathcal 
A)$. In the untwisted situation [MNW] the Hochschild groups $H_n(\mathcal A, \mathcal A)$ vanish for $n \geq 2$, in contrast to the classical situation $q = 1$. However, in the twisted situation, there exists an automorphisms $\rho$ with $H_n(\mathcal A, {}_{\rho}\mathcal A)\neq 0$ for $n = 0, 1, 2$. These automorphisms which are the positive powers of the canonical modular automorphism associated to the $SU_q(2)$-invariant linear functional is not an order two automorphism and hence will not be discussed here, but it is interesting to note this phenomenon studied in detail in \cite{H}. \par

To our knowledge, the Hochschild and periodic cyclic (co)homology of these two $\mathbb Z_2$ orbifolds of the Podle\'s two spheres are not known in the literature. In this article we compute the Hochschild and cyclic homology and cohomology groups of the two $S^2_{q,1}$ $\mathbb Z_2$-orbifolds $D_q$ and $\mathbb RP^2_q$. We also compute the Chern-Connes indices for each of these orbifolds by pairing the even periodic cocycles with the projections. While the projection for the quantum two dimensional real projective space was calculated in \cite{HMS}, for the quantum disc case the group $K_0(C(D_q)) \cong \mathbb Z$ as $C(D_q)$ is a Toeplitz algebra \cite[pp 191]{We}.

\section{$\mathbb Z_2$ actions on the Podle\'s sphere.}
The Podle\'s spheres $S^2_{q,s}$ has its C*-algebra $\mathcal A$, closure of the *-algebra generated by $A$ and $B$ satisfying the following relations:
$$A=A^\ast, \text{ } BA=q^2AB, \text{ }B^\ast B+A^2=(1-s^2)A + s^2, \text{ }BB^\ast + q^4A^2=(1-s^2)q^2A+s^2.$$
The automorphism of $\mathcal A$ acts diagonally on the generators, explicity for $\rho \in Aut(\mathcal A)$ and $\lambda \in \mathbb C$, $\rho$ has one of the following actions on the generators:

$$\sigma_\lambda(B)=\lambda B, \text{ } \sigma_\lambda(A)=A, \text{ } \sigma_\lambda(B^\ast) = \lambda^{-1} B^\ast.$$ 
$$\mu_\lambda(B)=\lambda B, \text{ } \mu_\lambda(A)=-A, \text{ } \mu_\lambda(B^\ast) = \lambda^{-1} B^\ast.$$ 
For $\rho = \sigma_{-1}$, the algebra $\mathcal A \rtimes_{\sigma_{-1}} \mathbb Z_2$ is associated to the quantum disc $D_q$, while for $\rho = \mu_{-1}$; the algebra $\mathcal A \rtimes_{\mu_{-1}} \mathbb Z_2$ corresponds to the quantum real projective space $\mathbb R P^2_q$\cite{HMS}.

\section{Strategy of the proof}
We use the paracyclic decomposition technique to decompose the the homology groups\cite{GJ}. 
One can use the results of \cite{GJ} to deduce the following decomposition of the homology group of the algebra $\mathcal A \rtimes \G$.
\begin{thm}\cite{GJ}
If $\G$ is finite and $|\G|$ is invertible in k, then there is a natural isomorphism of cyclic homology and
\begin{center}
$HH_\bullet(\mathcal A \rtimes \G)= HH_\bullet(H_0(\G, (\mathcal A)_\G^\sharp),$
\end{center}
where $(H_0(\G, (\mathcal A)_\G^\sharp)$ is the cyclic module
\begin{center}
$H_0(\G,(\mathcal A)_\G^\sharp)(n) =  H_0(\G,k[\G] \otimes (\mathcal A)^{\otimes {n+1}})$.
\end{center}
\end{thm} 
Since $\G$ is abelian, we can conclude that the group homology $H_0(\G, \mathcal {A} ^{ \natural}_{\G})$ splits the complex into $|\G|$ disjoint parts.
\begin{center} 
$H_0(\G,\mathcal {A} ^{ \natural}_{\G})(n) = H_0(\G, k[\G] \otimes (\mathcal {A}_{} )^{\otimes {n+1}}) = \displaystyle \bigoplus_{t \in \G} (({}_{t}\mathcal {A} )^{\otimes {n+1}})^{\G}$\\
\end{center}
For each $t \in \G$ ,  the algebra ${}_{t}\mathcal {A} $ is set-wise $\mathcal A$ with the twisted Hochschild differential ${}_{t}b$ acting as 
$${}_{t}b(a_0 \otimes a_1 \otimes \cdots \otimes a_n) = b'(a_0 \otimes a_1 \otimes \cdots \otimes a_n) + (-)^n((t \cdot a_n)a_)\otimes a_1 \otimes \cdots \otimes a_{n-1})$$
on the complex ${}_{t}\mathcal A^{\otimes (\bullet+1)}$.
We therefore decompose Hochschild homology $HH_\bullet(\mathcal A \rtimes \G)$ as follows:
\begin{center}
$HH_\bullet(\mathcal A \rtimes \G)= HH_\bullet(H_0(\G,\mathcal {A} ^{ \natural}_{\G}))=\displaystyle \bigoplus_{t \in \G} HH_\bullet(({}_{t}\mathcal {A}^{ \bullet})^{\G})$.
\end{center} 
It is enough to calculate $HH_\bullet(({}_{t}\mathcal {A} ^{ \bullet})^{\G})$ for each $t \in \G$. To calculate these individual homology groups, we use the lemma below.
\begin{lemma}\cite{Q}
Let 
\begin{center}
$J_\ast :=  0 \xleftarrow{d} A \xleftarrow{d} (\mathcal A^{\otimes2}) \xleftarrow{d} (\mathcal A^{\otimes3}) \xleftarrow{d} (\mathcal A^{\otimes4}) \xleftarrow{d} (\mathcal A^{\otimes5}) \xleftarrow{d} ...$
\end{center}
be a chain complex. For a given $\G$ action on $\mathcal A$, consider the following chain complex, with chain map $d^\G :(\mathcal A^{\otimes n})^\G \rightarrow (\mathcal A^{\otimes{n-1}})^\G$ induced from the map $d : \mathcal A^{\otimes n} \rightarrow \mathcal A^{\otimes{n-1}}$.
\begin{center}
$J^\G_{\ast} := 0 \xleftarrow{d^\G} \mathcal A^\G \xleftarrow{d^\G} (\mathcal A^{\otimes2})^\G \xleftarrow{d^\G} (\mathcal A^{\otimes3})^\G \xleftarrow{d^\G} (\mathcal A^{\otimes4})^\G \xleftarrow{d^\G} (\mathcal A^{\otimes5})^\G \xleftarrow{d^\G} ...$
\end{center}
With the $\G$ action commuting with the differential $d$. We have the following group equality $H_\bullet(J_\ast^\G, d^\G)=H_\bullet(J_\ast, d)^\G$.
\end{lemma}
Hence using the above lemma we have the following is the decomposition of the Hochschild homology group $H_\bullet(\mathcal A \rtimes_{\rho} \mathbb Z_2)$:
$$H_\bullet(\mathcal A \rtimes_{\rho} \mathbb Z_2) =H_\bullet(\mathcal A,\mathcal A)^{\rho} \oplus H_\bullet(\mathcal A, {}_{\rho}\mathcal A)^{\rho}. $$
Where ${}_{\rho}\mathcal A$, set-wise $\mathcal A$, is an $\mathcal A^e(=\mathcal A \otimes \mathcal A^{op})$ bi-module with the following actions:
$$\alpha \cdot a = (-1 \cdot \alpha)a\text{ and } a \cdot \alpha = a\alpha, \text{ for } \alpha \in \mathcal A \text{ and } a \in {}_{\rho}\mathcal A$$

The Hochschild and cyclic homology groups of $H_\bullet(\mathcal A, {}_{\rho}\mathcal A)$ and $HC_\bullet(\mathcal A, {}_{\rho}\mathcal A)$ equal the twisted Hochschild and cyclic homology groups $HH^{\rho}_{\bullet}(\mathcal A)$ and $HC^{\rho}_{\bullet}(\mathcal A)$ as $\rho$ is diagonal \cite{K}. Hence in order to understand $H_{\bullet}(\mathcal A \rtimes_{\rho} \mathbb Z_2)$ we need to understand the $\rho$ invariant subgroups of $H_\bullet(\mathcal A, \mathcal A)$ and $H_\bullet(\mathcal A, {}_{\rho}\mathcal A)$. For this we use the MNW resolution which we describe below.  \par

In the article \cite{MNW}, the authors presented a resolution of $\mathcal A$,
$$ \cdots \rightarrow \mathcal M_{n+1} \rightarrow \mathcal M_n \cdots \rightarrow \mathcal M_2 \rightarrow \mathcal M_1 \rightarrow \mathcal M_0 \rightarrow \mathcal A \rightarrow 0 $$ 
by free left $\mathcal A^e$-modules $\mathcal M_n$, with $rank(\mathcal M_0)=1$, $rank(\mathcal M_1)=3$ and $rank(\mathcal M_n)=4$ for $n\geq 2$. Adapting their notation, $\mathcal M_1$ has basis $\{e_A, e_B, e_{B^{\ast}}\}$, with $d_1: \mathcal M_1 \rightarrow \mathcal M_0 = \mathcal A^e$ given by
$$d_1(e_t) = t \otimes 1 - 1 \otimes t^o, \text{  } t = A,B,B^{\ast}.$$
The module $\mathcal M_2$ has basis $\{e_a \wedge e_B, e_A \wedge e_B^{\ast}, \vartheta^{(1)}_S, \vartheta^{(1)}_T \}$, with $d_2: \mathcal M_2 \rightarrow \mathcal M_1$ given by 
$$d_2(1_{\mathcal A^{e}} \otimes (e_A \wedge e_{B^{\ast}})) = (A\otimes 1 - 1 \otimes q^2A^{o})\otimes e_{B^{\ast}}-(q^2B^{\ast} \otimes 1 -1\otimes B^{\ast o})\otimes e_A,$$
$$d_2(1_{\mathcal A^{e}} \otimes (e_A \wedge e_{B^{\ast}})) = (q^2A\otimes 1 - 1 \otimes A^{o})\otimes e_{B^{}}-(B^{} \otimes 1 -1\otimes q^2B^{ o})\otimes e_A,$$
$$d_2(1_{\mathcal A^{e}} \otimes \vartheta^{(1)}_S) = -q^{-1}(B\otimes 1 \otimes e_{B^{\ast}}+1 \otimes B^{\ast o } \otimes e_B) - q(q^2(A\otimes  1 + 1 \otimes A^{o})\otimes e_A,$$
$$d_2(1_{\mathcal A^{e}} \otimes \vartheta^{(1)}_T) = -q^{-1}(1 \otimes B^o \otimes  e_{B^{\ast}} +B^{\ast} \otimes 1 \otimes e_B) - q^{-1}((A\otimes  1 + 1 \otimes A^{o})\otimes e_A.$$
The maps $d_i$; $i \geq 3$ are not needed in this article, for reference purpose we refer to \cite{MNW}. \par

\begin{center}
\begin{tikzcd}
    \arrow{r} & \mathcal M_2\arrow{r}{d_2}\arrow[harpoon]{d}{f_2} & \mathcal{M}_1\arrow{r}{d_1}\arrow[harpoon]{d}{f_1}  & \mathcal M_0\arrow{r}{d_0}\arrow[harpoon]{d}{f_0} & \mathcal A\arrow{r}\arrow{d}{\cong} & 0 \\
    \arrow{r} & \mathcal A^{\otimes 4}\arrow{r}{b'}\arrow[harpoon]{u}{h_2}\arrow{r} & \mathcal A^{\otimes 3}\arrow[harpoon]{u}{h_1}\arrow{r}{b'} & \mathcal A^{\otimes 2}\arrow{r}{b'}\arrow[harpoon]{u}{h_0}\arrow{r} & \mathcal A\arrow{r} & 0
\end{tikzcd}
\end{center}

\
There exists chain homotopy equivalence between the MNW and the bar resolution of $\mathcal A$. This is a standard construction in homological algebra. Let $(\mathcal M, d)$ denote the MNW resolution of $\mathcal A$ and let $(\mathcal N, b')$ be the bar resolution. Then we have $\{\mathcal M\}_i = \mathcal M_i$ and $\{\mathcal N\}_i = \mathcal A^{\otimes {(i+2)}}$. And the maps $\{f\}_{i \geq 0} : \mathcal M_i \rightarrow \mathcal A^{\otimes(i+2)}$ lift the identity map on $\mathcal A$ onto the complex between the resolutions. Simialrly let $\{h\}_{i \geq 0} :  \mathcal A^{\otimes(i+2)} \rightarrow \mathcal M_i$ be the chain homotopy equivalence map between the resolutions. Explicitly these maps are as follows:
\begin{center}
$ f_0(a_1 \otimes a_2^o) = (a_1,a_2)$,   $f_1(e_t) = (1,t,1)$; for $t=A,B,B^{\ast}$.
\end{center}
The maps $f_i$, $i \geq2$ can be found inductively satisfying the relation $f_i \cdot d_{i+1} = b' \cdot f_{i+1}$.

Since  a Poincar\'e-Birkhoff-Witt(PBW) basis for $\mathcal A$ consists of the monomials
$$\{A^kB^j\}_{j,k\geq0}, \{A^kB^{\ast (j+1)}\}_{j,k\geq 0}$$
 we define $h_1$ on the above PBW basis elements. 
For $a,b \in \mathcal A$, $h_0(a , b) =a \otimes b^o$ and for $t = B, B^\ast$
\begin{center}
$h_1(a,A^nt^m, b)= (a \otimes b^o)\{(t^m)^o(A^{n-1})^oe_A + (t^m)^o(A^{n-2})^oAe_A+...+ (t^m)^oA^{n-1}e_A+A^n(t^{m-1})^oe_t+A^n(t^{m-2})^ote_t+...+A^nt^{m-1}e_t\}$.
\end{center}
In order to locate the $\mathbb Z_2$ invariant cocycles of $H^\bullet(\mathcal A, {}_{\rho}\mathcal A)$, we need to use the resolution homotopy maps
$$h_{\ast}: \mathcal N_\ast \to \mathcal M_\ast$$ and $$f_{\ast}:\mathcal M_\ast \to \mathcal N_\ast.$$
We push a cocycles $\mathcal D$ into the bar complex and let $\mathbb Z_2$ act on it. Then, in the MNW complex, we compare the pullback of this $\mathbb Z_2$-acted cocycle with $\mathcal D$ to check the $\mathbb Z_2$ invariance.

\section{The quantum disc}
For $\rho =\sigma_{-1}$, the algebra $\mathcal A \rtimes_{\sigma{-1}} \mathbb Z_2$ corresponds to the quantum disc $D_q$. 
\begin{thm}[Hochschild and Cyclic Homology] The Hochschild and cyclic homology groups of $D_q$ are as follows:
\begin{center}
$H_\bullet(C(D_q),C(D_q) ) \cong\begin{cases}
\mathbb C^{\mathbb N} & \text{ for } \bullet  = 0,1\\
0 & \text{ for } \bullet >1. \end{cases}$\\
\end{center}
\begin{center}
$HP_\bullet(C(D_q)) \cong\begin{cases}
\mathbb C^4 & \text{ for } \bullet  = 2n\\
0 & \text{ for } \bullet = 2n+1. \end{cases}$\\
\end{center}
\end{thm}
\begin{proof}
We have the following :
$$H_\bullet(C(D_q), C(D_q)) =H_\bullet(\mathcal A,\mathcal A)^{\sigma_{-1}} \oplus H_\bullet(\mathcal A, {}_{\sigma_{-1}}\mathcal A)^{\sigma_{-1}}. $$
From \cite{H} we know that $H_0(\mathcal A,\mathcal A) = \mathbb C[1] \oplus \mathbb C[A] \oplus \Sigma_{m \geq 1}^{\oplus}\mathbb C[B^m]\oplus \Sigma_{m \geq 1}^{\oplus}\mathbb C[B^{\ast m}]$. We push each of these cyclces into the bar complex using the map $f_0$ and let $\sigma_{-1}$ act on them, thereafter we pull it back to the MNW complex using the map $h_0$. We see that:
$$H_0(\mathcal A,\mathcal A)^{\sigma_{-1}} = \mathbb C[1] \oplus \mathbb C[A] \oplus \Sigma_{m \geq 1}^{\oplus}\mathbb C[B^{2m}]\oplus \Sigma_{m \geq 1}^{\oplus}\mathbb C[B^{\ast 2m}].$$
Similarly we also get that
$$H_0(\mathcal A,{}_{\sigma_{-1}}\mathcal A)^{\sigma_{-1}} = \mathbb C[1] \oplus \mathbb C[A] .$$
Hence we have $H_0(C(D_q),C(D_q)) \cong \mathbb C^{\mathbb N}$. Similarly we compute that $H_1(C(D_q),C(D_q) )$,  we use maps $h_1$ and $f_1$ to locate the invariant cyclic cocycles in the groups $H_\bullet(\mathcal A,\mathcal A)^{\sigma_{-1}}$ and $H_\bullet(\mathcal A, {}_{\sigma_{-1}}\mathcal A)^{\sigma_{-1}}$. The group $H_1(C(D_q),C(D_q) )$ is generated by the elements $1 \otimes {}_{\sigma_{-1}}{e_A}$, $1 \otimes e_A$, $\{B^{2j+1} \otimes e_B\}_{j\geq0}$ and $\{B^{\ast {2j+1}} \otimes e_{B^{\ast}}\}_{j\geq0}$. Here ${}_{\sigma_{-1}}{e_A}$ denotes the invariant copy of the cycle $e_A \in H_0(\mathcal A,{}_{\sigma_{-1}}\mathcal A)$. Since for $\rho= id \text{ and } \sigma_{-1}$ and  $\bullet >1$, $H_\bullet(\mathcal A, {}_{\rho}\mathcal A)  = 0$. We have
$$H_\bullet(C(D_q),C(D_q) ) = 0 \text{ for } \bullet>1.$$
In a similar way we compute $HC_\bullet(C(D_q),C(D_q) )$. We observe that the cyclic homology group  $HC^{\rho}_{2n}(\mathcal A) = \mathbb C[1] \oplus \mathbb C[A]$ and $HC^{\rho}_{2n+1}(\mathcal A) =0$ for $\rho \in \{\sigma_{-1} , id\}$ \cite{H}. Using the paracyclic decomposition for the cyclic homology $HC_{\bullet}(C(D_q))$ we have:
$$HC_\bullet(C(D_q)) =HC_\bullet(\mathcal A,\mathcal A)^{\sigma_{-1}} \oplus HC_\bullet(\mathcal A, {}_{\sigma_{-1}}\mathcal A)^{\sigma_{-1}}. $$
We now check that each of the four cycles of $HC_{2n}(\mathcal A,\mathcal A) \oplus HC_{2n}(\mathcal A, {}_{\sigma_{-1}}\mathcal A)$ are $\sigma_{-1}$ invariant.
\end{proof}
\begin{cor}[Hochschild and Cyclic Cohomology] The Hochschild and cyclic homology groups of $D_q$ are as follows:
\begin{center}
$H^\bullet(C(D_q),C(D_q )^{'}) \cong\begin{cases}
\mathbb C^{\mathbb N} & \text{ for } \bullet  = 0,1\\
0 & \text{ for } \bullet >1. \end{cases}$\\
\end{center}
\begin{center}
$HP^\bullet(C(D_q)) \cong\begin{cases}
\mathbb C^4 & \text{ for } \bullet  = 2n\\
0 & \text{ for } \bullet = 2n+1. \end{cases}$\\
\end{center}
\end{cor}
\begin{proof}
Using the universal coefficient theorem we have relation between the Hochschild homology and the cohomology with dual algebra as the coefficient, $H_\bullet(C(D_q),C(D_q))^{'} = H^\bullet(C(D_q), C(D_q)^{'})$ \cite{L}. Hence $H^\bullet(D_q,D_q^{'}) = 0$ for $\bullet > 1$. And similarly we conclude that for $\bullet  =0, 1$; Hochschild cohomology groups are countably infinite in dimension. For periodic cyclic cohomology we consider the $B,S,I$ long exact sequence for Hochschild and cyclic cohomology and use that fact that higher Hochschild cohomology groups vanish.
\begin{center}
$\cdots \rightarrow HH^1({}_{\sigma_{-1}}\mathcal A )^{\sigma_{-1}} \xrightarrow{B} HC^0({}_{\sigma_{-1}}\mathcal A )^{\sigma_{-1}} \xrightarrow{S} HC^2({}_{\sigma_{-1}}\mathcal A )^{\sigma_{-1}} \xrightarrow{I} HH^2({}_{\sigma_{-1}}\mathcal A )^{\sigma_{-1}} \xrightarrow{B}  HC^1({}_{\sigma_{-1}}\mathcal A )^{\sigma_{-1}} \xrightarrow{S} \cdots$
\end{center} 

Since $HC^{\rho}_{2n}(\mathcal A) = k[1] \oplus k[A]$ for $\rho = \{id, \sigma_{-1}\}$\cite[Prop. 5.2]{H}, and the above spectral sequence stabilizes with all further maps being zero, we conclude that the group $HP^{even}(C(D_q),C(D_q)) \cong \mathbb C^4$ and is generated by $[\tau_0]$, $[f_A]$ and $[{}_{\sigma_{-1}}\tau_0]$ and $[{}_{\sigma_{-1}}f_A]$, where for $\tau_0(1)=1$ and $\tau_0(a) =0$ for all $a \in \mathcal A$, and $f_A$ is the Haar state on $\mathcal A(SU_q(2))$ restricted to the Podle\'s sphere and is given by $f_A(A^{r+1}) = (1-q^4)(1-q^{2r+4})^{-1}$ and for $s>0$ and it vanishes on the PBW basis elements $A^r B^ s$ and $A^r B^{\ast  s}$;  $f_A(A^r B^ s) =0=f_A(A^rB^{\ast s})$\cite{SW}. For $a \in \mathcal A$ the element ${}_{\sigma_{-1}}a \in {}_{\sigma_{-1}}\mathcal A$ and hence ${}_{\sigma_{-1}}\tau_0(1_{{}_{\sigma_{-1}}\mathcal A})=1$ and ${}_{\sigma_{-1}}\tau_0(a)=0$ for $a\in {}_{\sigma{-1}}\mathcal A$. And likewise we defined the Haar measure ${}_{\sigma_{-1}}f_A$ on the $\sigma_{-1}$ copy of $\mathcal A$.
\end{proof}

\section{The quantum real projective space}
\begin{thm}[Hochschild and Cyclic Homology] The Hochschild and cyclic homology groups of $\mathbb RP^2_q$ are as follows:
\begin{center}
$H_\bullet(C(\mathbb RP^2_q),C(\mathbb RP^2_q) ) \cong\begin{cases}
\mathbb C^{\mathbb N} & \text{ for } \bullet  = 0,1\\
0 & \text{ for } \bullet >1. \end{cases}$\\
\end{center}
\begin{center}
$HC_\bullet(C(\mathbb RP^2_q),C(\mathbb RP^2_q) ) \cong\begin{cases}
\mathbb C^2 & \text{ for } \bullet  = 2n\\
0 & \text{ for } \bullet = 2n+1. \end{cases}$\\
\end{center}
\end{thm}
\begin{proof}
Similar to the quantum disc case, we have the following:
$$H_\bullet(C(\mathbb RP^2_q)) =H_\bullet(\mathcal A,\mathcal A)^{\mu_{-1}} \oplus H_\bullet(\mathcal A, {}_{\mu_{-1}}\mathcal A)^{\mu_{-1}}. $$
From \cite{H} we know that $H_0(\mathcal A,\mathcal A) = \mathbb C[1] \oplus \mathbb C[A] \oplus \Sigma_{m \geq 1}^{\oplus}\mathbb C[B^m]\oplus \Sigma_{m \geq 1}^{\oplus}\mathbb C[B^{\ast m}]$. We observe that:
$$H_0(\mathcal A,\mathcal A)^{\mu_{-1}} = \mathbb C[1] \oplus \Sigma_{m \geq 1}^{\oplus}\mathbb C[B^{2m}]\oplus \Sigma_{m \geq 1}^{\oplus}\mathbb C[B^{\ast 2m}].$$
Similarly we also have
$$H_0(\mathcal A,{}_{\mu_{-1}}\mathcal A)^{\mu_{-1}} = \mathbb C[1].$$
Hence we have $H_0(D_q,D_q ) \cong \mathbb C^{\mathbb N}$. For $\rho = \mu_{-1}$, $H_1(\mathcal A, {}_{\mu_{-1}}\mathcal A) = 0$ and hence $H_1(C(\mathbb RP^2_q),C(\mathbb RP^2_q) )$ is countably infinite dimensional generated by the elements $1 \otimes e_A$, $\{B^{2j+1} \otimes e_B\}_{j\geq0}$ and $\{B^{\ast {2j+1}} \otimes e_{B^{\ast}}\}_{j\geq0}$. Higher Hochschild homology groups $HH^{\rho}_{\bullet}({}_{\mu_{-1}}\mathcal A)$ for $( \bullet >1)$ vanishes\cite{H} and hence using the paracyclic decomposition for $H_\bullet(C(\mathbb RP^2_q),C(\mathbb RP^2_q) ) $ we conclude that:
$$H_\bullet(C(\mathbb RP^2_q),C(\mathbb RP^2_q) ) = 0 \text{ for } \bullet>1.$$

To compute $HC_\bullet(C(\mathbb RP^2_q),C(\mathbb RP^2_q) )$. We observe that the cyclic homology group  $HC^{\mu_{-1}}_{n}(\mathcal A) = 0$ for  all $n > 0$\cite{H}. Using the paracyclic decomposition for the cyclic homology $HC_{\bullet}(\mathbb RP^2_q)$ we have:
$$HC_\bullet(C(\mathbb RP^2_q)) =HC_\bullet(\mathcal A,\mathcal A)^{\mu_{-1}} \oplus HC_\bullet(\mathcal A, {}_{\mu_{-1}}\mathcal A)^{\mu_{-1}}. $$
We now check that both the cycles of $HC_{2n}(\mathcal A,\mathcal A)$ are $\mu_{-1}$ invariant. Hence the result.
\end{proof}
\begin{cor}[Hochschild and Cyclic Cohomology] The Hochschild and cyclic homology groups of $\mathbb RP^2_q$ are as follows:
\begin{center}
$H^\bullet(C(\mathbb RP^2_q),C(\mathbb RP^2_q) ) \cong\begin{cases}
\mathbb C^{\mathbb N} & \text{ for } \bullet  = 0,1\\
0 & \text{ for } \bullet >1. \end{cases}$\\
\end{center}
\begin{center}
$HP^\bullet(C(\mathbb RP^2_q)) \cong\begin{cases}
\mathbb C & \text{ for } \bullet  = 2n\\
0 & \text{ for } \bullet = 2n+1. \end{cases}$\\
\end{center}
\end{cor}
\begin{proof}
Through the reasoning as before we conclude that for $\bullet  =0, 1$; Hochschild cohomology groups are countably infinite in dimension and vanishes for $\bullet >1$. For periodic cyclic cohomology we have that the group $HP^{even}(C(\mathbb RP^2_q),C(\mathbb RP^2_q)) \cong \mathbb C$ and is generated by $\mathbb C[1]$. This is so because $H_\bullet(\mathcal A, {}_{\tau_{-1}}) = 0$ for $\bullet>0$ and Among the two cocycles of $HP^{even}(\mathcal A)$, only the one dimensional subspace spanned by $[1]$ is $\tau_{-1}$ invariant. Hence the result.
\end{proof}

\section{Chern-Connes Indices}
Chern-Connes indices are useful invariants in noncommutative geometry. Explicitly, for a C*-algebra $\mathcal B$ over $\mathbb C$ we have the following map\cite[Section 8]{L}:
$$ch_{0, {n}} : K_0(\mathcal B) \rightarrow HC_{2n}(\mathcal B).$$
defined by $$[e] \mapsto tr(c(e)).$$ 
where $K_0(\mathcal B)$ is the Grothendieck group of the ring $\mathcal B$. We have the following pairing:
$$K_0(\mathcal B) \times HC^{2n}(\mathcal B) \xrightarrow[]{ch \times id} HC_{2n}(\mathcal B) \times HC^{2n}(\mathcal B) \rightarrow \mathbb C.$$ 
In this section we calculate the above pairing for the quantum disc $D_q$ and the quantum real projective space $\mathbb RP^2_q$. The vanishing of the second Hochschild homology leaves the two orbifolds with fewer periodic cocycles than expected. While the algebra $C(D_q)$ is a Toeplitz algebra and hence $K_0(C(D_q)) \cong \mathbb Z$ and is generated by $[1_{D_q}]$. Since  $HC^{2n}(D_q) \cong \mathbb C^4$. Hence we have the following Chern-Connes index table for $D_q$.
 \begin{center}
 
  \begin{tabular}{c || c | c | c | c || }
    
     &   $S\tau_0$ & $Sf_A$ & $S{}_{\sigma_{-1}}\tau_0$ & $S{}_{\sigma_{-1}}f_A$ \\ \hline \hline
    $[1_{D_q}]$ & $1$ & $0$ & $0$ &$0$\\ \hline \hline
  
  \end{tabular}
\end{center}
Similarly we have a description of the group $K_0(\mathbb RP^2_q) \cong \mathbb Z \oplus \mathbb Z_2$ \cite{HMS}generated by $[1]$ and $[P]$. Hence we have the following Chern-Connes index table for $\mathbb RP^2_q$:
\begin{center}
 
  \begin{tabular}{c || c | c || }
    
     &   $S\tau_0$    \\ \hline \hline
    $[1_{\mathbb RP^2_q}]$ & $1$ \\ \hline \hline
    $[P]$ & $0$ \\ \hline \hline
  \end{tabular}.
\end{center}

\section{Conclusion}
We see that the quantum parameter $q$ does not appear in the Chern-Connes indices of both the $\mathbb Z_2$ orbifolds of the Podle\'s quantum spheres. This can be attributed to the vanishing of the second homology of the Podle\'s spheres. We see that similar computations for the Podle\'s spheres yielded more invariants \cite{W}, but the vanishing of several projection of the the Podle\'s spheres after $\mathbb Z_2$ action leaves few projections on the quotient space. 

\textbf{Acknowledgment}:
 I acknowledge the discussion with Dr. Ulrich Kr\"ahmer and Prof. Xiang Tang and their valuable comments and references for the same.


\end{document}